\documentclass[a4paper,12pt]{article}
\usepackage{a4wide}
\usepackage{amsmath}
\usepackage{amssymb}
\usepackage{amsthm}
\usepackage{latexsym}
\usepackage{graphicx}
\usepackage[english]{babel}
\usepackage{makeidx}
              
\newtheorem{dfn} [subsection]{Definition}
\newtheorem{obs} [subsection]{Remark}
\newtheorem{exm} [subsection]{Example}

\newtheorem{prop}[subsection]{Proposition}

\newtheorem{teor}[subsection]{Theorem}
\newtheorem{lema}[subsection]{Lemma}
\newtheorem{cor} [subsection]{Corollary}

\def\Ass{\operatorname{Ass}}

\def\reg{\operatorname{reg}}

\begin{document}

\selectlanguage{english}
\frenchspacing

\large
\begin{center}
\textbf{Regularity of quasi-symbolic and bracket powers of Borel type ideals}

Mircea Cimpoea\c s
\end{center}

\normalsize

\begin{abstract}
In this paper, we show that the regularity of the q-th quasi-symbolic power $I^{((q))}$ and the regularity of the q-th bracket power $I^{[q]}$ of a monomial ideal of Borel type $I$, satisfy the relations $\reg(I^{((q))})\leq q\reg(I)$, respectively $\reg(I^{[q]})\geq q\reg(I)$. Also, we give an upper bound for $\reg(I^{[q]})$.

\vspace{5 pt} \noindent \textbf{Keywords:} Monomial ideals, Borel type ideals, Mumford-Castelnuovo regularity.

\vspace{5 pt} \noindent \textbf{2000 Mathematics Subject
Classification:}Primary: 13P10, Secondary: 13E10.
\end{abstract}

\section*{Introduction}

Let $K$ be an infinite field, and let $S=K[x_1,\ldots,x_n],n\geq 2$ the polynomial ring over $K$.
Bayer and Stillman \cite{BS} note that Borel fixed ideals $I\subset S$ satisfy the following property:
\[(*)\;\;\;\;(I:x_j^\infty)=(I:(x_1,\ldots,x_j)^\infty)\;for\; all\; j=1,\ldots,n.\] Herzog, Popescu and Vladoiu \cite{hpv} define a monomial ideal $I$ to be of \emph{Borel type} if it satisfies $(*)$. We mention that this concept appears in \cite[Definition 1.3]{CS} as the so called {\em weakly stable ideal}. Also, this concept appears in \cite[Definition 3.1]{BG}, as the so called \emph{monomial ideal of nested type}. We further studied this class of monomial ideals in \cite{mircea} and \cite{mir}.

In the first section, we recall some results regarding ideals of Borel type. Also, we discuss the relation between the sequential chain of an ideal of Borel type $I$, defined in \cite{hpv}, and the primary decomposition of $I$. 

Let $I\subset S$ be a monomial ideal and $I=\bigcap_{i=1}^r Q_i$ the an irreduntant primary decomposition of $I$, obtained in a canonical way. We define $I^{((q))}:=\bigcap_{i=1}^r Q_i^q$, the $q$-th \emph{quasi-symbolic} power of $I$, see Definition $2.1$. We prove that if $I$ is an ideal of Borel type, then $I^{((q))}$ and $I^{[q]}$ are also ideals of Borel type, where $I^{[q]}=(u^q\;:\; u\in I $ monomial) is the $q$-th bracket power of $I$. 

In \cite{mir}, we proved that $\reg(I^q)\leq q \reg (I)$. We give a similar result for the $q$-th quasi-symbolic power. More precisely, we prove that $\reg(I^{((q))})\leq q\reg(I)$, see Theorem $2.4$. Also, we prove that $\reg(I^{[q]})\geq q \reg(I)$, see Theorem $2.6$. In Proposition $2.11$, we prove that $\reg(I^{[q]})\leq q\reg(I)+(q-1)(n-1)$.

\footnotetext[1]{The support from grant ID-PCE-2011-1023 of Romanian Ministry of Education, Research and
Innovation is gratefully acknowledged.}

\pagebreak
\section{Some basic facts on Borel type ideals.}

Firstly, we recall the following equivalent characterizations of ideals of Borel type given in \cite{hpv} and in \cite{BG}.

\begin{prop}
Let $I\subset S$ be a monomial ideal. The following conditions are equivalent:

(a) $I$ is an ideal of Borel type.

(b) For any $1\leq j<i \leq n$, we have $(I:x_i^{\infty}) \subset (I:x_j^{\infty})$.

(c) Each $P\in Ass(S/I)$ has the form $P=(x_1,\ldots,x_m)$ for some $1\leq m\leq n$.



\end{prop}

Let $I\subset S$ be a monomial ideal of Borel type. Since each prime ideal $P\in Ass(S/I)$ is of the form $P=(x_1,\ldots,x_m)$ for some $1\leq m\leq n$, we can assume that $I$ has an irredundant primary decomposition:
\begin{equation}
 \label{1}
 I=\bigcap_{i=1}^r Q_{i};\;\;such\;that\;P_{i} := \sqrt{Q_{i}} = (x_1,\ldots,x_{n_{i-1}}),\;\;n\geq n_0 > n_1 > \cdots > n_{r-1}\geq 1.\;
\end{equation}

For each $0 \leq i \leq r-1$, we define $I_{i}:=\bigcap_{j=i+1}^{r} Q_{j}$. We claim that $I_{i+1}=(I_{i}:x_{n_{i}}^{\infty})$ for all $0\leq i \leq r-1$. Indeed, since $Q_{i+1}$ is $P_{i+1}$-primary, it follows that there exists a positive integer $k$ such that $x_{n_{i}}^k \in Q_{i+1}$. So
$(I_{i}:x_{n_i}^{\infty}) \supseteq ((Q_{i+1}\cdot I_{i+1}): x_{n_i}^{\infty}) \supseteq (x_{n_{i}}^k \cdot I_{i+1} : x_{n_i}^{\infty}) = I_{i+1}$. For the converse inclusion, note that $(I_{i}:x_{n_i}^{\infty}) \subseteq 
(Q_{i+1}:x_{n_i}^{\infty})\cap (I_{i+1}:x_{n_i}^{\infty}) = S \cap I_{i+1} = I_{i+1}$.

Thus, the chain of ideals $I = I_0 \subset I_1 \subset \cdots \subset I_{r-1}\subset I_r := S$ is the \emph{sequential chain} of $I$, as it was defined in \cite{hpv}. Note that $n_i=\max\{j:\;x_j|u\;$ for some $u\in G(I_i) \}$, where we denoted by $G(I_i)$ the set of minimal monomial generators of $I_i$.

Let $J_{i}$ be the monomial ideal generated by $G(I_{i})$ in $S_{i}:=K[x_1,\ldots,x_{n_{i}}]$, $0\leq i\leq r$. Then, the saturation $J_{i}^{sat} = (J_{i}:\mathbf{m}_{i}^{\infty})$ is generated by the elements of $G(I_{i+1})$, where $\mathbf{m}_{i}=(x_1,\ldots,x_{n_i})S_i$. It follows that
$ I_{i+1}/I_{i} \cong (J_{i}^{sat} / J_{i})[x_{n_{i}+1},\ldots,x_n]$. 

It would be appropriate to recall the definition of the Castelnuovo-Mumford regularity. We refer the reader to \cite{E} for further details on the subject.

\begin{dfn}
Let $K$ be an infinite field, and let $S=K[x_1,\ldots,x_n]$, $n\geq 2$ the polynomial ring over $K$. Let $M$ be a finitely generated graded $S$-module. The Castelnuovo-Mumford regularity $\reg(M)$ of $M$ is
\[ \max_{i,j} \{ j-i:\; \beta_{ij}(M)\neq 0 \}, \] 
where $\beta_{ij}(M)=dim_K (Tor_i(K,M))_j$ denotes the $ij$-th graded Betti number of $M$.
\end{dfn}

If $M=\bigoplus_{t\geq 0}M_t$ is an artinian graded $S$-module, we denote $s(M)=\max\{t:\;M_t\neq 0\}$. Herzog, Popescu and Vl\u adoiu
proved the following formula for the regularity of a monomial ideal of Borel type:

\begin{prop}\cite[Corollary 2.7]{hpv} If $I$ is a Borel type ideal, with the notations above, we have
\[ \reg(I)=\max\{ s(J_0^{sat}/J_0), \ldots, s(J_{r-1}^{sat}/J_{r-1})\} + 1. \]
\end{prop}


\begin{exm}
We consider the ideal $Q=(x_1^{a_1},\ldots,x_m^{a_m})\subset S$, where $1\leq m\leq n$ and $a_1\geq a_2 \geq \cdots \geq a_m \geq 1$. According to Proposition $1.3$, $\reg(Q)=s(\bar{S}/\bar{Q})+1$, where $\bar{S}=K[x_1,\ldots,x_m]$ and $\bar{Q}=\bar{S}\cap Q$. Since $u=x_1^{a_1-1}\cdots x_m^{a_m-1}\in \bar{S}$ is the monomial of the highest degree which is not contain in $\bar{Q}$, it follows that
\[ \reg(Q) = \sum_{i=1}^{m} (a_i-1) + 1 = a_1+\cdots+a_m - m + 1. \]
We consider the ideal $Q^q=(x_1^{qa_1},\ldots,x_m^{qa_m}, x_1^{(q-1)a_1}x_2^{a_2}, \ldots )$. Note that $Q^q\cap \bar{S} = \bar{Q}^q$ and therefore $\reg(Q^q)=s(\bar{S}/\bar{Q}^q)+1$. One can easily see that 
$u = x_1^{qa_1-1}x_2^{a_2-1}\cdots x_m^{a_m-1}$ is the monomial of the highest degree which is not contain in $\bar{Q}^q$. Thus:
\[ \reg(Q^q) = qa_1 - 1 + \sum_{i=2}^m (a_i-1) + 1 = qa_1+a_2+\cdots +a_m - m +1.\]
Note that $\reg(Q^q)\leq q\reg(Q)$, as we already know from \cite[Corollary 1.8]{mir}, and the equality holds if and only if $a_2=\cdots=a_m=1$.
\end{exm}

\section{Regularity of quasi-symbolic and bracket powers of Borel type ideals}

Now, assume $I\subset S$ is an arbitrary monomial ideal. Then $I$ has a unique irreduntant decomposition $I=\bigcap_{i=1}^s C_i$, where $C_i$ are irreducible monomial ideals. One obtains from this presentation a
\emph{canonical presentation} of $I$ as an intersection of primary ideals, $I=\bigcap_{i=1}^r Q_i$, where each $Q_i$ is $P_i$-primary and is defined to be the intersection of all $C_j$'-s with $\sqrt{C_j}=P_i$. See \cite{mono} for further details.

\begin{dfn}
Let $q$ be a positive integer. We define the \emph{$q$-th quasi-symbolic power of $I$} to be the ideal 
$$ I^{((q))}:=\bigcap_{i=1}^r Q_i^q.$$
\end{dfn}

Note that, $I^{(q)} \subset I^{((q))}$, where $I^{(q)}:= S\cap \bigcap_{P\in \Ass(S/I)}I^qS_{P}$ is the \emph{$q$-th symbolic power} of $I$. The equality holds if all $P_i$'-s are pairwise incomparable, but, in general, this is not the case. On the other hand, $I^q\subset I^{(q)}$. 

Now, assume $I\subset S$ is of Borel type with the primary decomposition $(1)$. One can easily see that $I^qS_{P_1}\cap S = I^q$, since all the minimal monomial generators of $I$ are from $K[x_1,\ldots,x_{n_0}]$ and $P_1=(x_1,\ldots,x_{n_0})$. Therefore, $I^{(q)}=I^q$.

In the following, we will assume that the primary decomposition $(1)$, of a Borel type ideal $I\subset S$, is canonical in the above sense. We have the following lemma.

\pagebreak

\begin{lema}
If $I\subset S$ is an ideal of Borel type and $q$ is a positive integer, then $Ass(S/I^{((q))})\subset Ass(S/I)$.
In particular, $I^{((q))}$ is an ideal of Borel type.
\end{lema}

\begin{proof}
Assume $I=\bigcap_{i=1}^r Q_i$ is the primary decomposition of $I$ given in $(1)$. It follows that
$I^{((q))}:=\bigcap_{i=1}^r Q_i^q$. This primary decomposition of $I^{((q))}$ is not necessarily irredundant. However,
since $\sqrt{Q_i^q} = \sqrt{Q_i}$, it follows that $Ass(S/I^{((q))})\subset Ass(S/I)$. Therefore, by Proposition $1.1(c)$, $I^{((q))}$ is an ideal of Borel type.
\end{proof}

\begin{exm}
We consider the following ideals, $Q = (x^8, x^6y^2, x^2y^6, y^8)\subset S:=K[x,y,z]$, $Q'=Q+(x^4y^4)\subset S$, and $I:=(Q,z^2)\cap Q' = (Q, x^4y^4z^2) \subset S$. Since, $Q\subsetneq Q'$, it follows that $(Q,z^2)\cap Q'$ is a primary decomposition of $I$ and thus $Ass(S/I)=\{ (x,y), (x,y,z)\}$.

We have $Q=(x^8,y^2)\cap (x^6,y^6)\cap (x^2,y^8)$ and $Q'=(x^8,y^2)\cap (x^4,y^6)\cap (x^6,y^4) \cap (x^2,y^8)$. Therefore, $I = Q' \cap (x^6,y^6,z^2)$ is the canonical primary decomposition of $I$, and thus $I^{((2))} = Q^2\cap (x^6,y^6,z^2)^2$. On the other hand,
\[ Q'^2 = Q^2 = (x^{16}, x^{14}y^2, x^{12}y^4, x^{10}y^6, x^8y^8,  x^6y^{10}, x^4y^{12}, x^2y^{14}, y^{16}), \]
and thus $I^{((2))} = Q^2$, since $Q^2 \subset (x^6,y^6,z^2)^2$. We have $s(K[x,y]/(Q'\cap K[x,y]))=8$ and $s(Q'/(Q,z^2x^4y^4)) = 11$, and therefore, by Proposition $1.3$, we get $\reg(I)=12$. Also, $s(K[x,y]/(Q\cap K[x,y])^2) = 16$ and thus $\reg(I^{((2))})=17$, according to Proposition $1.3$.
\end{exm}

Let $I\subset S$ be a Borel type ideal with the primary decomposition $I:=\bigcap_{i=1}^{r} Q_{i}$ from $(1)$.
We consider the sequential chain $I=I_0\subset I_1\subset \cdots \subset I_r=S$ of $I$, where 
$I_i:=\bigcap_{j=i+1}^{r} Q_{j}$. Note that $I_i^{((q))} := \bigcap_{j=i+1}^{r} Q_{j}^q$, since the previous primary decompositions of $I_i$-'s are canonical. We consider the following chain of ideals
\[ I^{((q))} = I_0^{((q))} \subset I_1^{((q))} \subset \cdots I_r^{((q))} = S.\]

 In the chain above, we may have some equalities.
Nevertheless, if we denote $J_{i}$ be the monomial ideal generated by $G(I_{i})$ in $S_{i}:=K[x_1,\ldots,x_{n_{i}}]$, we have
\[ I_{i+1}^{((q))} / I_{i}^{((q))} \cong ((J_{i}^{((q))})^{sat}/J_{i}^{((q))})[x_{n_i+1},\ldots,x_n]. \]

Also, the sequential chain of $I_i^{((q))}$ is obtain from the previous chain of ideal, by removing those ideals $I_i$ with $I_i=I_{i-1}$. Thus, by Proposition $1.3$,
\begin{equation}
\label{2}
\reg(I^{((q))}) = \max \{s((J_i^{((q))})^{sat}/ J_i^{((q))}),\;0\leq i\leq r-1 \}+1.
\end{equation}

Now, we are able to prove the following Theorem.

\begin{teor} 
With the above notations, we have
$\reg(I^{((q))})\leq q\cdot \reg(I)$.
\end{teor}

\begin{proof}
We fix $0\leq i \leq r-1$. Since $I_{i}:=\bigcap_{j=i+1}^{r} Q_{j}$, it follows that $J_{i}=\bigcap_{j=i+1}^{r} \bar{Q}_{j}$, where $\bar{Q}_{j}$ is the ideal generated by $G(Q_{j})$ in $S_{i}$. 
On the other hand, since $J_{i}^{sat}$ is generated by the elements of $G(I_{i+1})$, it follows that
$J_{i}^{sat} = \bigcap_{j=i+2}^{r} \bar{Q}_{j}$. Note that 
\[ s(J_{i}^{sat}/ J_{i}) + 1 = \min\{j:\; \mathbf{m}_{i}^j J_{i}^{sat} \subset J_{i} \} \]
and therefore
$ s(J_{i}^{sat}/ J_{i}) + 1 = \min\{j: \mathbf{m}_{i}^j \bar{Q}_{k}\subset \bar{Q}_{i+1} \;\;for\;\;all\;\;
k=i+2,\ldots,r\;\}. $
Analogously, since $I_{i}^{((q))}:=\bigcap_{j=i+1}^{r} Q_{j}^q$, it follows that
\[ s((J_{i}^{((q))})^{sat}/ J_{i}^{((q))}) + 1 = \min\{j: \mathbf{m}_{i}^j \bar{Q}_{k}^q \subset \bar{Q}_{i+1}^q \;\;for\;\;all\;\;k=i+2,\ldots,r\;\}. \]
Note that if $\mathbf{m}_{i}^j \bar{Q}_{k} \subset \bar{Q}_{i+1}$ then  
$\mathbf{m}_{i}^{jq} \bar{Q}_{k}^q = (\mathbf{m}_{i}^j \bar{Q}_{k})^q \subset \bar{Q}_{i+1}^q$. 
Therefore, we get 
\begin{equation}
\label{2}
s((J_{i}^{((q))})^{sat}/ J_{i}^{((q))}) + 1 \leq q \cdot (s(J_{i}^{sat}/ J_{i}) + 1).
\end{equation}
By applying Proposition $1.3$ to $I$ and $(3)$ we get the required conclusion.
\end{proof}

Let $I\subset S$ be a monomial ideal of Borel type. An interesting question is to find a relation between
$\reg(I^{q})$ and $\reg(I^{((q))})$.

Let $I\subset S$ be a monomial ideal and let $q$ be a nonnegative integer. We define the \emph{$q$-th bracket power} of $I$, to be the ideal $I^{[q]}$, generated by all monomials $u^q$, where $u\in I$ is a monomial. In particular, $I^{[0]}=S$ and $I^{[1]}=I$. Note that if $G(I) = \{ u_1,\ldots,u_m \}$ is the set of minimal monomial generators of $I$, then $G(I^{[q]}) = \{u_1^q,\ldots,u_m^q \}$. Note that $I^{[q]} \subset I^q$ for all $q$. In fact, when $q\geq 2$, the equality holds if and only if $I$ is principal. Also, one can easily see that $(I\cap J)^{[q]} = I^{[q]}\cap J^{[q]}$ for any monomial ideals $I,J\subset S$.

Now, assume $I=\bigcap_{i=1}^r Q_i$ is an irredundant primary decomposition of $I$. We claim that 
$I^{[q]} = \bigcap_{i=1}^r Q_i^{[q]}$ is an irredundant primary decomposition of $I^{[q]}$, where $q$ is a positive integer. In order to prove this, we fix an integer $i$ with $1\leq i\leq r$ and we chose a monomial 
$u\in Q_i \setminus \bigcap_{j\neq i} Q_j$. Obviously, $u^q\in Q_i^{[q]}$. We claim that $u^q\notin \bigcap_{j\neq i} Q_j$. Assume this is not the case. It follows that $u^q = u_j^q w_j$ for some monomials $u_j\in Q_j$ and $w_j\in S$, for all $j\neq i$. Therefore, $u_j|u$ for all $j\neq i$. It follows that $u\in \bigcap_{j\neq i} Q_j$, a contradiction.

As a consequence, we get the following Lemma.

\begin{lema}
If $I\subset S$ be a monomial ideal and $q$ a positive integer, then $Ass(S/I)=Ass(S/I^{[q]})$. 
In particular, if $I$ is of Borel type, then $I^{[q]}$ is of Borel type.
\end{lema}

Now, we are able to prove the following Theorem.

\begin{teor}
Let $I\subset S$ be a monomial ideal of Borel type. Then: $$\reg(I^{[q]})\geq q\cdot \reg(I).$$
\end{teor}

\begin{proof}
We consider the primary irredundant decomposition $\bigcap_{i=1}^rQ_i$ of $I$ from $(1)$ and the sequential chain $I = I_0 \subset I_1 \subset \cdots \subset I_r := S$ of $I$, where
$I_{i} = \bigcap_{j=i+1}^{r} Q_{j}$, for $0\leq i\leq r-1$. Note that the sequential chain of $I^{[q]}$, is 
$I^{[q]} = I_0^{[q]} \subset I_1^{[q]} \subset \cdots \subset I_r^{[q]} = S$. Indeed, all the inclusions are stricts.

We fix an integer $0\leq i \leq r-1$. Let $J_{i}$ be the monomial ideal generated by $G(I_{i})$ in $S_{i}:=K[x_1,\ldots,x_{n_{i}}]$. We denote $\bar{Q}_j$, the ideal generated by $G(Q_{j})$ in $S_{i}$, for all $1\leq j\leq r$. With these notations, we have $J_{i}=\bigcap_{j=i+1}^{r} \bar{Q}_{j}$ and 
$J_{i}^{[q]} = \bigcap_{j=i+1}^{r} \bar{Q}_{j}^{[q]}$. On the other hand, since $J_{i}^{sat}$ is generated by the elements of $G(I_{i+1})$, it follows that $J_{i}^{sat} = \bigcap_{j=i+2}^{r} \bar{Q}_{j}$.

Let $u\in J_{i}^{sat}\setminus J_{i}$ be a nonzero monomial. We claim that 
$x_1^{q-1}u^q \in (J_{i}^{[q]})^{sat}\setminus J_{i}^{[q]}$. It is clear that $x_1^{q-1}u^q \in (J_{i}^{[q]})^{sat}$. If we assume that $x_1^{q-1}u^q\in J_{i}^{[q]}$, it follows that $x_1^{q-1}u^q = v^q\cdot w$, where $v\in J_{i}$ is a monomial and $w\in S$ is a monomial. Since $v^q|x_1^{q-1}u^q$, it follows that $v|u$ and therefore $u\in J_{i}$, a contradiction.

As a consequence, we get $s((J_{i}^{[q]})^{sat} / J_{i}^{[q]}) \geq q\cdot s(J_{i}^{sat}/J_{i})+q-1$.
By applying Proposition $1.3$, we get the required conclusion.
\end{proof}

\begin{obs}
The conclusions of Theorem $2.4$ and Theorem $2.6$ hold for monomial ideals $I\subset S$ with $Ass(S/I)$ totally ordered by inclusion. Indeed, if $I$ is such an ideal, we can define a ring isomorphism $\varphi:S\rightarrow S$ given by a reordering of variables, such that $\varphi(I)$ is an ideal of Borel type. Since the Castelnuovo-Mumford regularity is an invariant, it follows that $reg(I)=reg(\varphi(I))$.
\end{obs}

Bermejo and Giemenez give in \cite{BG} a formula for the regularity of a Borel type ideal $I$, when the irredundant irreducible decomposition is known. More precisely, they proved the following Proposition.

\begin{prop}\cite[Corollary 3.17]{BG}
Let $I\subset S$ be a monomial ideal of Borel type. Assume $I=\bigcap_{i=1}^m C_i$ is the irredundant irreducible decomposition of $I$. Then:
\[ \reg(I) = \max \{reg(C_i):\; i=1,\ldots,m\}. \]
\end{prop}

Since $C_i$-'s are irreducible monomial ideals, they are generated by powers of variables. Since $\sqrt{C_i}\in Ass(S/I)$ and $I$ is of Borel type, we may assume that $C_i=(x_1^{a_{i1}}, \ldots, x_{r_i}^{a_{ir_i}})$, where $r_i$ is an integer with $1\leq r_i \leq n$ and $a_{ij}$ are some positive integers. Denote $S_i:=K[x_1,\ldots,x_{r_i}]$. If we denote $\bar{C}_i$ the ideal generated by $G(C_i)$ in $S_i$, then, by Proposition $1.3$, as in Example $1.5$, we have
$\reg(C_i) = s(S_i/\bar{C}_i)+1 = a_{i1} + \cdots + a_{ir_i} - r_i + 1$.
Therefore, we get the following corollary.

\begin{cor}
With the notations above, $$\reg(I)=\max\{a_{i1} + \cdots + a_{ir_i} - r_i + 1:\;i=1,\ldots,m\}.$$
\end{cor}

Let $q$ be a positive integer and consider the ideal $I^{[q]}$. Since $I=\bigcap_{i=1}^m C_i$, it follows that
$I^{[q]} = \bigcap_{i=1}^m C_i^{[q]}$ and $C_i^{[q]}= (x_1^{qa_{i1}}, \ldots, x_{r_i}^{qa_{ir_i}})$. Note that
$\bigcap_{i=1}^m C_i^{[q]}$ is the irredundant irreducible decomposition of $I^{[q]}$. Indeed, we can argue in the same way as we did for the irreducible primary decomposition of $I^{[q]}$. Therefore, by Corollary $2.9$, we get the following.

\begin{cor}
$reg(I^{[q]}) = \max\{ qa_{i1} + \cdots + qa_{ir_i} - r_i + 1:\; i=1,\ldots,m\}$.
\end{cor}

The above formula leads us to the following upper bound for $\reg(I^{[q]})$.

\begin{prop}
Let $I\subset S$ be an ideal of Borel type and let $q$ be a positive integer. Then:
\[ \reg(I^{[q]}) \leq q\reg(I) + (q-1)(n-1) = \alpha q \reg(I) - (n-1), \]
where $\alpha=1+\frac{n-1}{reg(I)}$.
\end{prop}

\begin{proof}
With the above notations, we may assume $\reg(I)=a_{i1} + \cdots + a_{ir_i} - r_i + 1$ for some $1\leq i\leq m$.
According to Corollary $2.9$ and Corollary $2.10$, $\reg(I^{[q]}) = qa_{i1} + \cdots + qa_{ii_r} - r_i + 1$. Therefore,
$\reg(I^{[q]}) = q\reg(I) + (q-1)(r_i-1)$. Since $r_i-1\leq n-1$, we get the required inequality. The remaining equality is trivial.
\end{proof}

We conclude our paper, with the following example.

\begin{exm}
Let $I=(x)\cap(x^2,y)=(x^2,xy)\subset S=K[x,y]$. Let $q$ be a positive integer. It follows that
$I^q = (x^{2q}, x^{2q-1}y, \ldots, x^qy^q) = (x^q)\cap(x^{2q},x^{2q-1}y,\ldots,x^{q+1}y^{q-1},y^q)$.

Also, we obtain $I^{((q))}=(x^q)\cap(x^2,y)^q = (x^q)\cap(x^{2q},x^{2q-2}y,\ldots,x^2y^{q-1},y^q) = \linebreak
(x^{2q}, x^{2q-2}y, \ldots ,x^{2q-2\left\lfloor \frac{q}{2} \right\rfloor}y^{\left\lfloor \frac{q}{2} \right\rfloor} , x^{q}y^{\left\lfloor \frac{q}{2} \right\rfloor+1} )$, where we denoted by $\left\lfloor \alpha \right\rfloor$ the integer part of $\alpha$. On the other hand, $I^{[q]}=(x^q)\cap(x^{2q},y^q)=(x^{2q},x^qy^q)$.

We consider the sequential chain of $I$, $I=:I_0\subset I_1 \subset I_2:=S$, where $I_1=(x)$. We have $J_0=I\subset S$ and $J_1=(x)\subset K[x]$. Therefore, $J_0^q=I^q$, $J_0^{((q))}=I^{((q))}$ and $J_0^{[q]}=I^{[q]}$. Also, $J_1^q=J_1^{((q))}=J_1^{[q]}=(x^q)\subset K[x]$. We get $J_0^{sat}=(x)S$, $(J_0^q)^{sat}=(J_0^{((q))})^{sat}=(J_0^{[q]})^{sat}=(x^q)S$ and
 $J_1^{sat} = (J_1^q)^{sat} = (J_1^{((q))})^{sat} = (J_1^{[q]})^{sat} = K[x]$.
 
We have $s(J_1^{sat}/J_1) = 0$ and $s((J_1^q)^{sat}/J_1^q) = s((J_1^{((q))})^{sat}/J_1^{((q))})= s((J_1^{[q]})^{sat}/J_1^{[q]}) = q-1$.

Also, one can easily compute $s(J_0^{sat}/J_0) = 1$, $s((J_0^q)^{sat}/J_0^q) = 2q-1$, $s((J_0^{((q))})^{sat}/J_0^{((q))}) = 2q-1$ and $s((J_0^{[q]})^{sat}/J_0^{[q]}) = 3q-2$. By Proposition $1.3$, it follows that $\reg(I)=2$, $\reg(I^{q})=\reg(I^{((q))})=2q$ and $\reg(I^{[q]})=3q-1$.

Since $I=(x)\cap(x^2,y)$ is also the irreducible irredundant decomposition of $I$, by Corollary $2.8$ and Corollary $2.9$, we can compute directly $\reg(I)=\max\{1-1+1, 2+1-2+1\} = 2$ and, respectively, $\reg(I^{[q]})=\max\{q-1+1, 2q+q - 2 + 1\} = 3q-1$. 

Note that $\reg(I^{[q]}) = q\reg(I) + (q-1)(2-1)$ and therefore, the upper bound given in Proposition $2.11$ is the best possible.
\end{exm}


\vspace{2mm} \noindent {\footnotesize
\begin{minipage}[b]{15cm}
 Mircea Cimpoeas, Simion Stoilow Institute of Mathematics of the Romanian Academy\\
 E-mail: mircea.cimpoeas@imar.ro
\end{minipage}}
\end{document}